\documentclass[a4paper,12pt]{article}
\usepackage{a4wide}
\usepackage{amsfonts}
\usepackage{graphics}
\usepackage{amsmath, amsthm}  
\usepackage{amssymb,bbm} 
\usepackage{enumerate}
\usepackage{amsthm}

\usepackage{tikz}
\usepackage{epstopdf}

\newcommand{\N}{\mathbb{N}}
\newcommand{\Q}{\mathbb{Q}}\newcommand{\Z}{\mathbb{Z}}\newcommand{\F}{\mathbb{F}}

\newcommand{\map}[3]{ #1 : #2 \longrightarrow #3 }

\newcommand{\mapl}[5]{ #1 : #2 \longrightarrow #3 : #4 \longmapsto #5 }

\newtheorem{theorem}{Theorem}[section]

\newtheorem{proposition}[theorem]{Proposition}

\newtheorem{corollary}[theorem]{Corollary}
\newtheorem*{corollary*}{Corollary}
\newtheorem*{fakecorollary}{Corollary}

\newtheorem{lemma}[theorem]{Lemma}

\newtheorem*{theorem*}{Theorem}
\newtheorem*{faketheorem}{Theorem}
\newtheorem*{proposition*}{Proposition}
\newtheorem*{lemma*}{Lemma}
\newtheorem*{problem*}{Problem}
\newtheorem{openproblem}[theorem]{Open Problem}
\newtheorem*{observation*}{Observation}

\theoremstyle{definition} \newtheorem{remark}[theorem]{Remark}
\theoremstyle{definition} \newtheorem{example}[theorem]{Example}
\theoremstyle{definition}\newtheorem{definition}[theorem]{Definition}

\title{Arithmetically-free group-gradings of Lie algebras: I}

\author{Wolfgang Alexander Moens\thanks{This research was supported by the Erwin Schr\"odinger Junior Fellowship (Grant XXXX), and the Austrian Science Foundation FWF (Grant J3371-N25).}}
\date{\today}

\begin{document}

\maketitle

\abstract{A Lie algebra $L$ is known to be nilpotent if it admits a grading by $(\Z_p,+)$ 
with support $X$ not containing $0$. It is also known that the class of $L$ can be bounded by some explicit function of $|X|$. We generalise this and other classical results to gradings of Lie algebras by arbitrary groups with arithmetically-free support. We then apply these results to automorphisms of groups satisfying an identity.}

\section{Introduction}

A well-known theorem of Thompson states that a finite group $(G,\cdot)$ is nilpotent if it admits an automorphism that has prime order $p$ and fixes only the neutral element, \cite{Thompson}. Kegel later extended Thompson's result by showing that a finite group $(G,\cdot)$ is nilpotent if it admits a splitting automorphism $f$ of prime order $p$: $$1_G = x \cdot f(x) \cdot f^2(x) \cdots f^{p-1}(x),$$ for all $x \in G$, \cite{Kegel}. Either result gives a positive answer to Frobenius' conjecture on transitive permutation groups. Thompson's result was also improved by Higman, who showed that $G$ must even be nilpotent of $p$-bounded class, \cite{Higman}. He did this by reducing the nilpotency of $G$ to the nilpotency of a Lie ring $L$ (naturally associated with $G$) admitting a regular $(\Z_p,+)$-grading. Higman proved:


\begin{theorem} To each prime $p$ corresponds a (minimal) integer $h(p)$ such that if a Lie ring $L$ has a grading $\bigoplus_{x \in \Z_p} L_{x}$ by $(\Z_p,+)$ such that $L_{\overline{0}} = \{ 0 \}$, then $L$ is nilpotent of class at most $h(p)$. \label{Higgie} \end{theorem}


It was not immediately clear how big the nilpotency class can be. (The minimal upper bound is conjectured to be $(p^2-1)/4$, for all primes $p >2$.) In \cite{Kreknin,KrekninKostrikin}, Kreknin and Kostrikin proved the explicit upper bound $$h(p) \leq \frac{(p-1)^{2^{p-1}-1}-1}{p-2}.$$ Khukhro then showed that the class of $L$ is bounded by a similar function of $|X|$, the cardinality of the support, \cite{KhukhroSupport}. Analogous results were obtained for Lie algebras that are graded by finite subsets of $(\Z^m,+) \setminus \{ 0 \}$. \newline

In order to generalise some of these results, we introduce the following definition. A finite subset $X$ of an abelian group $(G,+)$ is said to be \emph{arithmetically-free}, iff $X$ contains no full arithmetic progression $x,x+y,x+2y, \ldots$ with increment $y \in X$. A finite subset $X$ of an arbitrary group $(G,\cdot)$ is arithmetically-free, iff every abelian subset is arithmetically-free. So arithmetically-free sets can be thought of as natural generalisations of finite subsets of $\Z_p$ and $\Z^m$ not containing the neutral element. \newline

In \cite{MoensPreprint} we construct a map $\map{H}{\N}{\N}$ and prove the following generalisation of Higman's theorem:

\begin{theorem} Consider a Lie algebra $L$ that is graded by a group $G$. If the support $X$ is arithmetically-free, then $L$ is nilpotent of class at most $H(|X|)$. \end{theorem}


We note that there are no restrictions on the base field, the dimension of the algebra, or the structure of the grading group. 
In \cite{MoensPreprint}, it is shown that the hypothesis on $X$ is necessary: 

\begin{theorem} If a finite set $X$ is \emph{not} arithmetically-free, then it must support the grading of a non-nilpotent Lie algebra. \end{theorem}


The generalised Higman map $H$ is defined by an explicit recursion, and it grows very quickly. But if we suppose that the support also has a good-ordering (in the sense of Shalev), then we obtain a better bound. We shall prove:

\begin{faketheorem}[\ref{MainTheoremGrading}] Consider a Lie algebra $L$ that is graded by a group $G$. If the support $X$ is arithmetically-free and if $X$ has a good-ordering, then $L$ is nilpotent of class at most $1 + |X| + |X|^2 + \cdots + |X|^{2^{|X|-1}-2}$. \end{faketheorem}

The theorem is a consequence of more general results (proposition \ref{PropositionPermutationContraction}) about group-gradings of algebras satisfying the permutation-contraction property -- a property which is automatically satisfied if the algebra is (anti-)associative, alternative, or Lie. A similar bound can be obtained from a result about set-gradings satisfying an Engel-identity (proposition \ref{PropositionSetGrading}). \newline

The hypothesis of a good-ordering on $X$ is reasonable: it is automatically satisfied if the grading group $G$ is the multiplicative group of a field. 
Since the eigenspace decomposition of an automorphism is such a grading, we can extend classical results of Borel-Serre \cite{BorelSerre}, Higman, Kreknin-Kostrikin \cite{KrekninKostrikin}, Jacobson \cite{Jacobson}, and Khukhro \cite{KhukhroSupport} in Lie theory:

\begin{fakecorollary}[\ref{CorollaryGrading}] Consider a Lie algebra $L$ that is graded by the multiplicative group $(\mathbb{F}^\times,\cdot)$ of a field $\mathbb{F}$. If the support $X$ is arithmetically-free, then $L$ is nilpotent and its class is at most $1 + |X| + |X|^2 + \cdots + |X|^{2^{|X|-1}-2}$. \end{fakecorollary}

The corollary is about graded Lie algebras, but it can be used to describe the structure of groups with automorphisms satisfying an identity. 
By using standard results in the theory of $p$-adic analytic groups, as in Zel'manov's solution of the restricted Burnside problem, we obtain results of the form:


\begin{faketheorem}[\ref{TheoremAutomorphismP}] Consider a finitely-generated group $(G,\cdot)$ that is residually-(finite $p$). Suppose that the automorphism $f$ of $G$ satisfies an identity, that is: there exist constants $k \in \N$ and $m_1,\ldots,m_k,n_1,\ldots,n_k \in \Z$ such that for all $x \in G$, we have \begin{equation} 1_G = f^{m_1}(x^{n_1}) \cdot f^{m_2}(x^{n_2}) \cdots f^{m_k}(x^{n_k}). \tag{$\ast$} \end{equation} If the roots of the Laurent-polynomial $$r_{\F_p}(z) := \sum_{1 \leq t \leq k} n_t \cdot z^{m_t} \in \F_p[z,z^{-1}]$$ form an arithmetically-free subset of $(\overline{\F}_p^\times,\cdot)$, then $G$ is linear. If $G$ is also a torsion group, then $G$ is finite.
\end{faketheorem}

An identity $(\ast)$ for which the corresponding set of roots is arithmetically-free, is said to be arithmetically-free (over $\F_p$) of degree $m := \max_{1 \leq s,t \leq k} | m_s - m_t | $. This condition on the roots is technical, but easily satisfied. If $f$ is a split automorphism of prime order $q \neq p$, for example, 
then its associated Laurent-polynomial $r_{\F_p}(z)$ is the cyclotomic polynomial $\Phi_q(z)$, and its set of roots is arithmetically-free in $(\overline{\mathbb{\F}}_p^\times,\cdot)$. Moreover: the identity $(\ast)$ is arithmetically-free (over $\F_p$) if $r_{\F_p}(z)$ is an irreducible polynomial (over $\F_p$) and satisfies $r_{\F_p}(0) \cdot r_{\F_p}(1) \neq 0$, (Lemma \ref{LemmaIrreducible}). 


\begin{fakecorollary}[\ref{CorollaryBoundP}] Consider a $d$-generated $p$-group of finite exponent $l$. If an automorphism of $G$ satisfies an arithmetically-free identity of degree $m$ over $\F_p$, then 
$$|G| \leq l^{d^{m^{2^m}+1}}.$$ \end{fakecorollary}


We justify the condition on the exponent in Example \ref{ExampleIdentitiesAF}. Applying Corollary \ref{CorollaryGrading} to the torsion-free case is more straight-forward:

\begin{faketheorem}[\ref{TheoremTorsionFree}] Consider a finitely-generated, torsion-free, nilpotent group $(G,\cdot)$. If some automorphism of $G$ satisfies an arithmetically-free identity over $\Q$ of degree $m$, then $G$ is nilpotent and $c(G) \leq m^{2^m}.$ \end{faketheorem}


These results all come from gradings by the multiplicative group of a field. But if the algebra is graded by the \emph{additive} group $(\F,+)$ of a field $\F$ of prime characteristic, then the grading group will have highly non-trivial torsion, so that the question of nilpotency becomes more complicated. In section $4$ we will consider a special problem: for which $n\in \N$ and $p \in \mathbb{P}$ is the set of $n$'th roots of unity arithmetically-free in the \emph{additive} group $(\overline{\mathbb{F}}_p,+)$? We will illustrate how it is related to a classical result of Kostrikin-Kuznetsov and the contributions of Shalev-Zel'manov to the resolution of the co-class conjectures for $p$-groups. \newline

\emph{Convention:} Not all groups in this paper will be abelian, but an additive group $(G,+)$ is always understood to be abelian. Since we will sometimes be working with groups that come from a ring or field, we will try to avoid confusion by referring to the group together with its operation. For example: $(G,+) \subseteq (\mathbb{F}_p,+)$ and $(H,\cdot) \subseteq (\mathbb{F}_p^\times,\cdot)$. 

\section{Arithmetically-free subsets of groups}

For an element $g$ of a group $G$, we let $\langle g \rangle$ to be the cyclic subgroup of $G$ that is generated by $g$.

\begin{definition} For a subset $X$ of an abelian group $(G,+)$, we define its set of periods in $G$ to be $$P_G(X) := \{ g \in G | \exists x \in X : x + \langle g \rangle \subseteq X \}.$$ A \emph{finite} subset $X$ of an abelian group $(G,+)$ is said to be \emph{arithmetically-free} in $G$, iff $$X \cap P_G(X) = \emptyset.$$ A finite subset $X$ of an \emph{arbitrary} group $(G,\cdot)$ is arithmetically-free, iff for every abelian subgroup $H$ of $G$, the intersection $X_H := X\cap H$ is arithmetically-free in $H$: $$X_H \cap P_{H}(X_H) = \emptyset.$$ \end{definition}

We note that for a finite set $X$ in the abelian group $G$, the set of periods $P_G(X)$ is a union of finite cyclic subgroups $C$ of $G$, each of size $1 \leq |C| \leq |X|$.

\begin{proposition} \label{PropositionCharacterisationAF} A finite subset $X$ of an abelian group $(G,+)$ is arithmetically-free, iff $X$ satisfies one of the following two equivalent conditions:
\begin{enumerate}[a.]
\item If $X$ contains the truncated arithmetic progression $x,x+y,\ldots,x+ |X| \cdot y$, then $y \not \in X$,
\item If $X$ contains the full arithmetic progression $x, x+ y, x + 2y , \ldots$, then $y \not \in X$.
\end{enumerate}
More generally: a finite subset $X$ of an arbitrary group $(G,\cdot)$ is free, iff $X$ does not contain an arithmetic progression $x,x\cdot y,x \cdot y \cdot y,\ldots$ with $y \in X$ and $x \cdot y = y \cdot x$.
\end{proposition}

\begin{proof} We need only observe that $x,x+y,x+2y,\ldots,x+|X| y \subseteq X$, iff $x + \langle y \rangle \subseteq X$. \end{proof}

\paragraph{Examples of arithmetically-free sets.} Let us first consider some subsets which satisfy the condition trivially. We will show in section $4$ that they correspond with classical results in the literature.

\begin{example} \label{ExampleUnits} If $p$ is a prime, then $X := \Z_p^\times$ and all its subsets are arithmetically-free in $(\Z_p,+)$. There are no other arithmetically-free subsets of $\Z_p$. More generally: the set of generators $$\Z_n^\times := \{ x \in \Z_n | \operatorname{ord}_{(\Z_n,+)}(x) = n \}$$ of a non-trivial cyclic group $(\Z_n,+)$ is arithmetically-free in $(\Z_n,+)$. \end{example}

\begin{example} \label{ExampleTorsionFree} If $(G,\cdot)$ is a torsion-free group, then every finite subset $X$ of $G \setminus \{ 1 \}$ is arithmetically-free. There are no other arithmetically-free subsets of $G$. Special case: the groups $(\Z^m,+)$, for $m \in \N$. \end{example}

A combination of \emph{Example} \ref{ExampleUnits} and \ref{ExampleTorsionFree} gives: if every abelian subgroup of $G$ is simple or free-abelian, then every finite subset $X$ of $(G,\cdot) \setminus \{ 1 \}$ is arithmetically-free. The converse is also true. In particular: all Tarski monsters satisfy the property.

\begin{example} \label{ExampleOrder} If $X$ is a finite subset of a group $G$ satisfying $\operatorname{ord}_G(y) > |X|$ for all $y \in X$, then $X$ is an arithmetically-free subset of the group. In particular: suppose that $(\F,+,\cdot)$ is an algebraically closed field.
\begin{enumerate}[i.]
\item If $X \subseteq (\F^\times,\cdot)$ is a finite set that does not contain any roots of unity, then $X$ is arithmetically-free.
\item Suppose that $\F$ has characteristic $p > 0$. Let $n \in \N$ and let $\zeta$ be a primitive $n$'th root of unity. If $n < p$, then the set $\{ x \in \F | x^{n} = 1 \}$ of all $n$'th roots is arithmetically-free in the \emph{additive} (elementary abelian) group $(\mathbb{F}_p[\zeta],+)$.
\end{enumerate}
\end{example}

\begin{example} \label{ExampleSumFree} If $X$ is a finite sum-free subset of $(G,+)$, then $X$ is arithmetically-free. \end{example}

From these basic examples, we can generate more:

\begin{lemma} \label{LemmaProjectionAF} If $\map{\pi}{G}{H}$ is a surjective homomorphism of abelian groups, and if $Y$ is an arithmetically-free subset of $H$, then every finite subset of $X := \pi^{-1}(Y)$ is arithmetically-free in $G$. If $G$ is finite, then $\frac{|X|}{|G|} = \frac{|Y|}{|H|}$. \end{lemma}

\begin{proof} It suffices to note that $\pi(P_G(X)) \subseteq P_H(Y)$. \end{proof}

\begin{example} \label{ExampleDensity} If $(G,+)$ is finite and if the prime $p$ divides $|G|$, then $G$ has an arithmetically-free subset of density $1-\frac{1}{p} \geq \frac{1}{2}$. Indeed: there is a natural projection $\map{\pi}{G}{\Z_p}$, so that $X := \pi^{-1}(\Z_p^\times)$ is a maximal arithmetically-free subset with density $1 - \frac{1}{p}$ in $G$.\end{example}

This may be contrasted with the uniform upper bound for the density of sum-free subsets in finite abelian groups, \cite{GreenRuzsa}.

\paragraph{Growth of SubSums.} Let us generalise the elementary theorem of Cauchy and Davenport about sequences in $(\Z_p,+)$. For a sequence $(a_1,\ldots,a_k)$ on an abelian group $G$, we define
\begin{eqnarray*} \operatorname{SubSum}((a_1,\ldots,a_k))
& :=& \bigcup_{ \substack{\pi \in \operatorname{Sym}(k)}} \{ a_{\pi(1)} + \ldots + a_{\pi(k_0)} | 0 \leq k_0 \leq k \} \\
&=& \{ \sum_{i \in I} a_i | I \subseteq \{ 1,2,\ldots,k \} \}.
\end{eqnarray*}

The next lemma shows that the SubSum set of a sequence $S$ grows with the length of $S$, unless $P_G(\operatorname{SubSum}(S)) \cap X \not = \emptyset$.

\begin{lemma}[Growth or capture] \label{LemmaGrowth} Consider an abelian group $(G,+)$, a finite subset $X \subseteq G \setminus \{ 0 \}$, and a sequence $S := (a_1,\ldots,a_k)$ of length $k$ on $X$. Then either
\begin{enumerate}[a.]
\item $| \operatorname{SubSum}(S) | \geq k+1$, or
\item $ \operatorname{SubSum}(S) $ contains an arithmetic progression $a,a+b,a+2b, \ldots$ with increment $b \in X$.
\end{enumerate}
\end{lemma}

\begin{proof} The statement holds trivially for $k=1$, so we suppose that $k >1$ and we proceed by induction. Let $T$ be $(a_1,\ldots,a_{k-1})$ and note that $\operatorname{SubSum}(S) = \operatorname{SubSum}(T) \cup (\operatorname{SubSum}(T) + a_k)$. If $\operatorname{SubSum}(S)$ contains a progression with increment in $X$, we are done. Else, $\operatorname{SubSum}(T) \neq \operatorname{SubSum}(T) + a_k$ and $| \operatorname{SubSum}(T) | \geq k$, by the induction hypothesis. So $|\operatorname{SubSum}(S)| \geq k+ 1$, and we are done. \end{proof}

\begin{remark} We emphasise that we are studying the growth of SubSum for \emph{sequences}, (or equivalently: \emph{multi-sets}) rather than the growth of subset sums for \emph{sets}. \newline

In order to illustrate the difference we consider the constant sequence $S := (x,\ldots,x)$ of length $k$, for some $x \in \Z_p^\times$ and $k < p$. We see that the lower bound $k+1$ for $|\operatorname{SubSum}(S)|$ is attained. This is in stark contrast with the growth of subset sums for sets; a well-known result by Olson (which confirms a conjecture of Erdos and Heilbronn, \cite{Olson}) implies that every subset $X$ of $\Z_p^\times$ with density $\frac{|X|}{|\Z_p|} > 2 / \sqrt{p}$ is \emph{complete}, that is: $\operatorname{SubSum}(X) = \Z_p$. This result has recently been generalised by Vu to the subset $X := \Z_n^\times$ 
of the group $(\Z_n,+)$, \cite{Vu}. 
\end{remark}

The following proposition expresses the property of being arithmetically-free in terms of integer-valued invariants.

\begin{proposition}[Quantitative definition of arithmetically-free sets] \label{PropositionQuantitative} For a finite subset $X$ of an abelian group $(G,+)$, the following two properties are equivalent:
\begin{enumerate}[a.]
\item $\exists \nu \in \{ 1,2,\ldots, |X| \}$ such that $\forall a,b \in X$, $\exists \nu_0 \in \{1,2,\ldots,\nu\}$, such that $a + \nu_0 b \in G \setminus X$,
\item $\exists \delta \in \{ 1,2,\ldots, |X| \}$ such that $\forall a,b_1,\ldots,b_\delta \in X$, $\exists I \subset \{1,2,\ldots,\delta \}$ such that $a + \sum_{i \in I} b_i \in G \setminus X$. 
\end{enumerate}
We define $\delta(X)$ to be the minimal such $\delta$, and $\nu(X)$ to be the minimal such $\nu$. We have $$1 \leq \nu(X) \leq \delta(X) \leq |X|.$$
\end{proposition}

\begin{proof} The first property is clearly equivalent to the definition of arithmetic-freedom, and we obtain it from the second by specialisation: $b_1 = \cdots = b_{\delta}$. Now suppose $X$ is arithmetically-free. It then suffices to prove the second property for $\delta := |X|$. Let $a,b_1,\ldots,b_{|X|} \in X$ and define $S := (b_1,\ldots,b_{|X|})$. If $a + \operatorname{SubSum}(S) \subseteq X$, then $|\operatorname{SubSum}(S)| = |a + \operatorname{SubSum}(S)| \leq |X|$, and lemma \ref{LemmaGrowth} implies that $X$ is not arithmetically free. So we may conclude that $a + \operatorname{SubSum}(S)$ is not contained in $X$, and we are done. \end{proof}

\begin{example} Let $(G,+)$ be an abelian group.
\begin{enumerate}[i.]
\item The arithmetically-free subsets $X$ of $G$ satisfying $\nu(X) = \delta(X) = 1$, are precisely the finite sum-free subsets of $G$.
\item Every arithmetically-free subset $X$ of $(G,+)$ satisfying $\nu(X) = \delta(X) = |X|$, is a truncated arithmetic progression of the form $X := \{a,2a,\ldots,|X| a \}$ with $\operatorname{ord}(a) > |X|$. But if $G$ has torsion, the converse need not be true: the set $X := \{\overline{1},\overline{2},\overline{3}\}$ is not arithmetically-free in $(\Z_4,+)$, even though $\operatorname{ord}(\overline{1}) = \operatorname{ord}(\overline{3}) > |X|$. \item For the arithmetically-free subsets $X \subseteq G$ and $Y \subseteq H$ from lemma $1$, the densities $|X|/|G|$ and $|Y|/|H|$ agree. We also have: $\nu(X) = \nu(Y)$ and $\delta(X) = \delta(Y)$. 
\end{enumerate}
\end{example}


\paragraph{Good-orderings on arithmetically-free sets.} In order to prove the solvability of graded groups and rings, Shalev introduced the notion of good-orderings, \cite{Shalev}. This idea goes back to at least Kreknin.

\begin{definition} Let $X$ be a subset of an abelian group $(G,+)$. A total order $<$ on $X$ is a \emph{good-ordering}, iff there are no $x,y \in X$ with $x + y \in X$ and $x < x + y < y$. A subset $X$ of an arbitrary group $(G,\cdot)$ is good-ordered, iff every abelian subset of $X$ has a good-ordering. \end{definition}

We refer to the literature for more about the following theorem: \cite{BorelMostow,Kreknin,Shalev}.

\begin{theorem}[Borel-Mostow; Kreknin; Shalev] \label{TheoremShalev} Consider a Lie algebra $L$ that is graded by an abelian group $(G,+)$ with finite support $X$. If $X$ admits a good-ordering, and $0 \not \in X$, then $L$ is solvable of length at most $2^{|X|-1}-1$. \end{theorem}

It is not too difficult to generate examples that trivially satisfy the property.

\begin{lemma} A finitely-generated abelian group $(G,+)$ admits a good-ordering, iff the torsion subgroup of $G$ is cyclic. \end{lemma}

\begin{proof} Let $<$ be a good-ordering on a finite group $(T,+)$. Since the reverse $<'$ order defined by $a <' b \leftrightarrow b < a$ is a good-ordering as well, we may suppose that the minimal element $t_{\min}$ of $T$ is not $0$. We suppose that $T \neq \langle t_{\min} \rangle$, and define $u := \max \left( T \setminus \langle t_{\min} \rangle \right).$ The good-ordering then implies the contradiction: $$t_{\min} < u < u + t_{\min} \in \langle t_{\min} \rangle.$$ We conclude that every good-ordered finite group $(T,+)$ is cyclically generated by its minimal or maximal element. More generally: the good-orderings are given by $$t < 2 t < \cdots < |T| t = 0 \text{ for } T = \langle t \rangle,$$ or their reverse orderings. \newline

Now suppose that $G$ admits a good-ordering. Then also its finite torsion subgroup $T$ admits a good-ordering. The above then implies that $T$ is cyclic. Conversely, suppose that the torsion subgroup of $G$ is cyclic. Then $G$ has a good-ordering because it can be (abstractly) embedded into the circle group (lemma $2.2$ of \cite{Shalev}).\end{proof}

\begin{example} \label{ExampleMultiplicativeField} Every finite subset $X$ of the multiplicative group $(\F^\times,\cdot)$ of a field $\F$ admits a good-ordering. Indeed: the torsion subgroup of $(\langle X \rangle,\cdot) \subseteq (\F^\times,\cdot)$ is known to be cyclic. \end{example}

As before, we can lift these basic examples to generate more examples:

\begin{lemma} \label{LemmaProjectionGO} If $\map{\pi}{G}{H}$ is a surjective homomorphism of abelian groups and if $Y \subseteq H \setminus \{ 0 \}$ admits a good-ordering, then every subset of $X := \pi^{-1}(Y) \subseteq G$ admits a good-ordering. \end{lemma}

\begin{proof} Let $K\subseteq G$ be the kernel of the projection. For each coset $x + K$ that projects into $Y$, we fix an arbitrary total order on $x + K$. If $x_1,x_2$ project onto distinct elements of $Y$, then we define $x_1 < x_2$ iff $\pi(x_1) < \pi(x_2)$. We now have a total order on $X$ that is also a good-ordering. \end{proof}

\begin{example} Every finite subset $X$ of an abelian group $G$ that avoids a subgroup $N$ of prime index (i.e. $X \cap N = \emptyset$), is arithmetically-free and good-ordered. \end{example}

So if a prime $p$ divides the order of a finite abelian group $G$, then $G$ has a good-ordered, arithmetically-free subset $X$ of density $1 - \frac{1}{p}$, (cf. example \ref{ExampleDensity}). \newline

A positive answer to the following problem would allow us to remove the condition of a good-ordering in theorem $4$.

\begin{openproblem} Does every arithmetically-free subset $X$ of an abelian group $(G,+)$ admit a good-ordering? \end{openproblem}

\section{Gradings with arithmetically-free support}

\paragraph{The finite-dimensional case.} Suppose $L$ is a \emph{finite}-dimensional Lie algebra and suppose that it admits a grading $\bigoplus_{g \in G} L_g$ with arithmetically-free support $X$. Then all homogeneous elements $y \in L_g$ are $\operatorname{ad}$-nilpotent of bounded degree: $$\operatorname{ad}_y^{|X|}(L) = \{ 0 \} .$$ Equivalently: all homogeneous elements satisfy the $|X|$-Engel property. So we may apply Jacobson's theorem about weakly-closed sets of nilpotent operators (for finite-dimensional vector spaces) to conclude that $L$ is nilpotent. Unfortunately, Jacobson's theorem gives only the trivial upper bound for the class ($c(L) < \dim(L)$), and 
it cannot be applied if $L$ is infinite-dimensional. We shall try to overcome these obstacles by relating the solvable and nilpotent series of $L$: $$ L^{(n)} \subseteq L^{2^n} \text{ and } L^{1+|X|^m} \subseteq L^{(m)}.$$ 

\paragraph{Gradings with the permutation-contraction property.} In this section we will be considering set-gradings of algebras with the permutation-contraction property. They are directly inspired by the weakly-closed sets of operators introduced by Jacobson, by the group-gradings of $(\alpha,\beta,\gamma)$-algebras introduced by Bergen and Grzeszczuk, by the polynomial products of Radjavi, and by the $(\Z_n,+)$-gradings of Lie-type algebras introduced by Bakhturin and Zaicev, and used by Makarenko: \cite{JacobsonWeaklyClosed,BergenGrzeszczuk,Radjavi,BahturinZaicev,Makarenko}. \newline

Let us introduce some notation. For elements $b,c$ in an algebra $A$, we define $\operatorname{alg}(b,c)$ to be the subalgebra of $A$ that is generated by $b$ and $c$. For a subset $B$ of $A$, we define $\langle B \rangle_A := \langle B \rangle$ to be the smallest right-sided ideal of $A$ that contains $B$.

\begin{definition} Consider an algebra $A$ over a commutative ring $K$ and a finite set $Y$. We say the decomposition $A = \bigoplus_{y \in Y} A_y $ of $A$ into $K$-submodules is a grading of $A$ by $Y$, iff there exists a map $\map{f}{Y \times Y}{Y}$ such that for all $x,y \in Y$ we have $$A_x \cdot A_y \subseteq A_{f(x,y)}.$$ We say it satisfies the \emph{permutation-contraction} property, iff: for all homogeneous elements $a,b,c$ of $A$, there exist $\gamma \in K$ and $z \in \operatorname{Span}_K\{ z_1 \cdot z_2 | z_1,z_2 \in \operatorname{alg}(b,c) \}$, such that: $$ (a \cdot b) \cdot c = \gamma (a \cdot c) \cdot b + a \cdot z.$$ \end{definition}

Every group-grading is a set-grading in the obvious way, but not every set-grading is a group-grading.

\begin{example} If $A$ is one of the following algebras, then every grading of $A$ satisfies the permutation-contraction property: (anti-)associative algebras, right-alternative algebras, right-Leibniz (Loday)-algebras, and in particular, Lie algebras. \end{example}

We recursively define the solvable and nilpotent series, as usual: $A^1 := A := A^{(0)}$, $A^{n+1} := A^{n} \cdot A$, and $A^{(n)} := A^{(n-1)} \cdot A^{(n-1)}$, for $n \in \N$. If a grading $\bigoplus_{x} A_x$ satisfies the permutation-contraction property, then each $A^{(n)}$ has the obvious grading satisfying the permutation-contraction property. Also: $(A^{(1)})^n \cdot A \subseteq (A^{(1)})^n$.

\paragraph{Arithmetically-free group-gradings with the permutation-contraction property.} For convenience, we shall be using the left-associative convention for products: $x_1 \cdot x_2 \cdot x_3 := ( x_1 \cdot x_2 ) \cdot x_3$, and for $k > 3$, we define $x_1 \cdot \ldots \cdot x_k \cdot x_{k+1} := ( x_1\cdot \ldots \cdot x_k ) \cdot x_{k+1}$. 

\begin{proposition} \label{PropositionPermutationContraction} Consider a group-grading $\bigoplus_{g \in G} A_x$ of an algebra $A$ with the permutation-contraction property. If the group is abelian and if the support $X$ is arithmetically-free, then $$A^{1 + \delta(X)^s} \subseteq A^{(s)}.$$\end{proposition}

Let us prove the slightly stronger statement $A^{1 + \delta(X)^0 + \delta(X)^1 + \cdots + \delta(X)^{s-1}} \subseteq A^{(s)}$.

\begin{proof} For $s = 1$, the statement is true by definition: $A^{1 + 1} := A^{(1)}$. So let us proceed by induction on $s > 1$. For $h_1,h_2,g_1,\ldots,g_{t} \in X$ and $t \in \N$, we consider the homogeneous left-associative word $W := a_{h_1} \cdot a_{h_2} \cdot a_{g_1} \cdot \ldots \cdot a_{g_t}$ in the non-trivial homogeneous elements $a_{h_1} \in A_{h_1},\ldots,a_{g_t} \in A_{g_t}$. For any $\pi \in \operatorname{Sym}(t)$, we may also consider the twisted product $$W^{\pi} := a_{h_{1}} \cdot a_{h_2} \cdot a_{g_{\pi(1)}} \cdot \ldots \cdot a_{g_{\pi(t)}}.$$ The permutation-contraction-property then implies that $W$ is in the span of $W^{\pi}$ and $(A \cdot A) \cdot (A \cdot A)$. If $t \geq \delta$, then proposition \ref{PropositionCharacterisationAF} implies that there exists a permutation $\pi \in \operatorname{Sym}(t)$ and a cut-off $t_0 \leq t$ such that $$h_1 + h_2 + g_{\pi(1)} + \cdots g_{\pi(t_0)} \in G \setminus X.$$ In particular: $W^{\pi}$ is in the right-sided ideal generated by $A_{G \setminus X} = \{ 0 \}$, and therefore $W^{\pi} = 0$. We conclude that $W \subseteq A^{(1)} \cdot A^{(1)}$. More generally: for $t := l \cdot \delta$, we obtain $l$ permutations, and we see that $$A^{2 + l \delta} \subseteq (A^{(1)})^{l+1}.$$ Define $l := 1 + \delta + \cdots + \delta^{s-2}$. Then $1 + \delta + \cdots + \delta^{s-1} = 1 + \delta l$, and we may apply the induction hypothesis: $$A^{1 + 1 + \delta l} \subseteq (A^{(1)})^{1 + l} \subseteq (A^{(1)})^{(s-1)} := A^{(s)}. \qedhere$$ \end{proof}

The proof is very similar to the one given by Kreknin and Kostrikin for the special case of example \ref{ExampleUnits} (the group of prime order).

\paragraph{Set-gradings with the $m$-Engel property.} We will again be using the left-associative convention for Lie brackets: $[x_1,x_2,x_3] := [[x_1,x_2],x_3]$, and more generally: $[x_1,\ldots,x_{n+1}] := [[x_1,\ldots,x_n],x_{n+1}]$.

\begin{proposition} \label{PropositionSetGrading} Let $m,n,s \in \N$. Consider a grading $\bigoplus_{y \in Y} L_y$ of a Lie algebra $L$ by a \emph{set} $Y$ of size $|Y| = n$. Suppose that either of the following two claims holds.
\begin{enumerate}[(a)]
\item All the homogeneous components $L_y$ act $m$-nilpotently: $$\operatorname{ad}_{L_y}^m(L) = \{ 0 \}.$$
\item All the homogeneous elements $v_y$ are $m$-Engel, $$\operatorname{ad}_{v_y}^m(L) = \{ 0 \}$$ and the characteristic of $L$ does not divide $m!$.
\end{enumerate}
Then $$L^{1 + n^s \cdot m^{s}} \subseteq L^{(s)}.$$
\end{proposition}

\begin{proof} Define $r := n \cdot (m-1) + 1$.  Let us prove the slightly stronger statement $$L^{1 + r^0 + r^1 + \cdots r^{s-1}} \subseteq L^{(s)}.$$ Choose any total order $y_1 < y_2 < \cdots < y_n$ on the elements of $Y$. The lexicographical order on all finite sequences on $Y$ is then a total order. This gives us a partial order on the (left-associative) homogeneous words: $W < W'$, iff the corresponding sequences $S$ and $S'$ on $Y$ satisfy $S < S'$. \newline

\emph{Claim:} For every $a \in L$ and $y \in Y$, we have $\operatorname{ad}_{L_y}^m (a) \subseteq \langle \operatorname{ad}_{L^{(1)}}(a) \rangle.$ \newline

In case $(a)$ there is nothing to prove. In case $(b)$, we may use the standard linearisation argument: for all $b_1,\ldots,b_m \in L_y$ we have
$$ 0 = \sum_{\pi \in \operatorname{Sym}(m)} [a , b_{\pi(1)} , \cdots , b_{\pi(m)}] ,$$ so that $m! [ a , b_1 , \cdots  , b_m] \in \langle \operatorname{ad}_{L^{(1)}}(a) \rangle.$ \newline

\emph{Claim:} For every $a \in L$ and $s \in \N$, we have $\operatorname{ad}_{L}^{s r}(a) \subseteq \langle \operatorname{ad}_{L^{(1)}}^s(a) \rangle.$ \newline

Suppose the result is not true for $s := 1$. Then there exists a homogeneous word $W := [a , b_1 , \cdots , b_r] \not \in \langle \operatorname{ad}_{L^{(1)}}(a) \rangle$, with $b_1 \in L_{y_1},\ldots,b_r \in L_{y_r}$. We may suppose that $W$ is maximal under these conditions. If there are $m$ consecutive $y_j$ that are equal, we may apply the previous claim to conclude that $W \in \langle \operatorname{ad}_{L^{(1)}}(a) \rangle$. This contradicts the choice of $W$. So there exist $y_j < y_{j+1}$. Then $W < W^{(j,j+1)} \in \langle \operatorname{ad}_{L^{(1)}}(a) \rangle$, and we may apply the permutation-contraction property:
$$W \in \F W^{(j,j+1)} + \langle \operatorname{ad}_{L^{(1)}}(a) \rangle \subseteq \langle \operatorname{ad}_{L^{(1)}}(a) \rangle.$$ This again contradicts the choice of $W$. We conclude that $\operatorname{ad}_{L}^{r}(a) \subseteq \langle \operatorname{ad}_{L^{(1)}}(a) \rangle$. \newline

Let us now induct on $s > 1$. Consider an arbitrary homogeneous word $$W := [a , b_1 , \cdots , b_{(s-1)r} , d_1 , \cdots , d_r] := [V , d_1 , \cdots , d_r] ,$$ with $V := [a , b_1 , \cdots , b_{(s-1)r}]$. The induction hypothesis corresponding with $s-1$, applied to $V$ gives $$V \subseteq \langle \operatorname{ad}_{L^{(1)}}^{s-1}(a) \rangle,$$ so that the base of the induction applied to $W$ implies $$W \subseteq \langle \operatorname{ad}_{L^{(1)}}(V) \rangle \subseteq \langle \operatorname{ad}_{L^{(1)}}^s(a) \rangle.$$

\emph{Claim:} For every $s \in \N$, we have $L^{2 + s r} \subseteq (L^{(1)})^{1 + s}.$ \newline

This follows from the previous claim if we specialise $a := [a_1 , a_2]$. \newline

\emph{Claim:} For every $s \in \N$, we have $L^{1 + r^0 + \cdots + r^{s-1}} \subseteq L^{(s)}.$ \newline

For $s:=1$, the claim is true by definition. So we may suppose that $s > 1$ and use induction. Note that $1 + r^0 + \cdots + r^{s-1} = 2 + r(r^0 +\cdots + r^{s-2})$. The previous claim, combined with the induction hypothesis then gives: $$ L^{1 + r^0 + \cdots + r^{s-1}} \subseteq (L^{(1)})^{1 + r^0 +\cdots + r^{s-2}} \subseteq (L^{(1)})^{(s-1)} : = L^{(s)}.\qedhere$$
\end{proof}

\begin{remark} If we specialise $n := 1$, then we recover Higgins' theorem for $m$-Engel algebras, \cite{Higgins}. \end{remark}

\paragraph{Abelian subgroups of the grading group.} Consider a grading $\oplus_g L_g$ of a Lie algebra $L$. For a subset $H$ of $(G,\cdot)$ we define $L_H$ to be the homogeneous subspace of $L$ spanned by the elements $L_h$, with $h \in H$. (If $H$ is a subgroup of $G$, then $L_H$ will even be a homogeneous subalgebra of $L$.) For a group $(G,\cdot)$, we let $\operatorname{Ab}(G)$ be the family of all abelian subgroups of $G$.

\begin{lemma}[Reduction to abelian subgroups] \label{ReductionAbelian} Consider a grading $\oplus_{g \in G} L_g$ of a Lie algebra $L$ over a field $\F$ by an arbitrary group $(G,\cdot)$. Then for every $n \in \N$ we have $$L^n = \operatorname{Span}_{\F} \{ L_A^n | A \in \operatorname{Ab}(G) \} \text{ and } L^{(n)} = \operatorname{Span}_{\F} \{ L_A^{(n)} | A \in \operatorname{Ab}(G) \}. $$ In particular: if $L_A^n = \{ 0 \}$ for all $A \in \operatorname{Ab}(G)$, then $L^n = \{ 0 \}$. Similarly: if $L_A^{(n)} = \{ 0 \}$ for all $A \in \operatorname{Ab}(G)$, then $L^{(n)} = \{ 0 \}$. \end{lemma}

The proof is straightforward, but we include it for completeness.

\begin{proof} Let us first use induction on $n \in \N$ to prove the implication $$\{ 0 \} \neq [L_{g_1},\ldots,L_{g_n}] \implies g_1,\ldots,g_n \text{ commute pairwise}.$$ If $g_1$ and $g_2$ are elements of $G$, we may consider the commutator subspaces $[L_{g_1},L_{g_2}] \subseteq L_{g_1 \cdot g_2}$ and $[L_{g_2},L_{g_1}] \subseteq L_{g_2 \cdot g_1}$ of $L$. Since the bracket of $L$ is anti-commutative, we have $[L_{g_1},L_{g_2}] , [L_{g_2},L_{g_1}] \subseteq L_{g_1 \cdot g_2} \cap L_{g_2 \cdot g_1}$. If $g_1 \cdot g_2 \neq g_2 \cdot g_1$, then $L_{g_1 \cdot g_2} \cap L_{g_2 \cdot g_1} = \{ 0 \}$, so that $[L_{g_1},L_{g_2}] = 0 = [L_{g_2},L_{g_1}]$. For $n > 2$ we use to induction hypothesis to conclude that $g_1,\ldots,g_{n-1}$ commute pairwise. So it suffices to show that $g_n$ commutes with each of the $g_1,\ldots,g_{n-1}$. Since $$\{ 0 \} \neq [L_{g_1},\ldots,L_{g_n}] \subseteq [L_{g_1 \cdot g_2},L_{g_3},\ldots,L_{g_n}],$$ the induction hypothesis guarantees that $g_1 \cdot g_2,g_3,\ldots,g_{n-1}$ commute with $g_n$. The Jacobi-identity implies that there exists a permutation $\pi \in \operatorname{Sym}(n-1)$ such that $$\{ 0 \} \neq [L_{g_n},L_{g_{\pi(1)}},\ldots,L_{g_{\pi(n-2)}},L_{g_{\pi(n-1)}}].$$ The induction hypothesis implies that $g_n$ commutes with the elements of $V =: \{ g_{\pi(1)},\ldots,g_{\pi(n-2)} \}$. Since $V \cap \{ g_1,g_2\} \neq \emptyset$, we may assume without loss of generality that $g_2 \in V$. In particular: $g_2 \cdot g_n = g_n \cdot g_2$. Since $ g_1 \cdot g_n \cdot g_2 = (g_1 \cdot g_2) \cdot g_n = g_n \cdot (g_1 \cdot g_2)$, we obtain $g_1 \cdot g_n = g_n \cdot g_1$. This finishes the induction. \newline

We now note that $L^n$ is spanned by the homogeneous subspaces $[L_{g_1},\ldots,L_{g_n}]$. By the above, $L^n$ is spanned by the subspaces $[L_{g_1},\ldots,L_{g_n}]$ for which $g_1,\ldots,g_n$ all commute. This proves the first claim. The second claim follows from the first and the fact that any (not necessarily standard) Lie bracket in the spaces $L_{g_1}, \ldots, L_{g_n}$ is a linear combination of brackets $[L_{g_{\pi(1)}},\ldots,L_{g_{\pi(n)}}]$, where $\pi$ runs over the elements of $\operatorname{Sym}(n)$.
\end{proof}

\paragraph{Main results.}

\begin{theorem} Consider a group-graded Lie algebra $L$. If the support $X$ is arithmetically-free and if $X$ has a good-ordering, then $L$ is nilpotent and $$c(L) \leq \delta(X)^0 + \delta(X)^1 + \cdots + \delta(X)^{2^{|X|-1}-2} \leq |X|^{2^{|X|-1}-1}.$$ \label{MainTheoremGrading} \end{theorem}

\begin{proof} Lemma \ref{LemmaReductionAbelian} allows us to assume that the grading group is abelian. Since $X$ is arithmetically-free, the neutral element of the group does not belong to $X$. Since $X$ has a good-ordering, we may apply Shalev's theorem to conclude that $L$ is solvable of length at most $2^{|X|-1}-1$. Proposition \ref{PropositionPermutationContraction} then gives the first inequality. The second inequality follows from $\delta(X) \leq |X|$ of proposition \ref{PropositionCharacterisationAF}. \end{proof}

\begin{remark} If we replace proposition \ref{PropositionPermutationContraction} with proposition \ref{PropositionSetGrading} in the proof, and observe that $\nu(X) \leq |X|$, we obtain the upper bounds $$c(L) \leq (\nu(X) \cdot |X|)^{2^{|X|-1}-1} \leq |X|^{2^{|X|}-2}.$$ \end{remark}

We recall that every finite subgroup of the multiplicative group $(\F^\times,\cdot) := (\F \setminus \{ 0 \},\cdot)$ of a field $\F$ is \emph{cyclic} (cf. example \ref{ExampleMultiplicativeField}). This implies that every finitely-generated subgroup of $(\F^\times,\cdot)$ has a good-ordering, so that we obtain:

\begin{corollary} Consider a Lie algebra $L$, and suppose that it is graded by the multiplicative group $(\F^\times,\cdot)$ of a field $\F$. If the support $X$ is arithmetically-free, then $L$ is nilpotent, and $$c(L) \leq \delta(X)^{2^{|X|-1}-1} \leq |X|^{2^{|X|-1}-1}.$$ \label{CorollaryGrading}\end{corollary}

In view of the Khukhro-Makarenko-Shumyatsky theorem of \cite{KhukhroMakarenkoShumyatsky}, it makes sense to describe the structure of Lie algebras that admit a group-grading that is almost-arithmetically-free in the following sense:

\begin{openproblem} Consider a group-graded Lie algebra $L = \bigoplus_{g \in G} L_g$ and a subset $X$ of $G$. If $X$ is arithmetically-free and good-ordered, then $L$ has a nilpotent ideal $N$ of $|X|$-bounded class and $(|X|,\dim(L_{G \setminus X}))$-bounded codimension. \end{openproblem}

\section{Interpretation}

In this section we will apply the theorems (corollary) to the examples of section one.

\paragraph{Regular automorphisms of prime order.} Suppose that a Lie ring $L$ admits a periodic automorphism $\alpha$ of prime order $p$. After extending the scalars, the eigenspace decomposition of $L$ with respect to $\alpha$ is a $(\Z_p,+)$-grading. If the automorphism is also fix-point-free, then the support $X$ of the grading is contained in $\Z_p^\times$. So we have arrived at \emph{Example \ref{ExampleUnits}}. \newline

If $L$ is finite-dimensional over a field of characteristic zero, then the nilpotency follows from \cite{BorelSerre}. If $L$ is finite-dimensional over a field of arbitrary characteristic, then the nilpotency of $L$ follows from \cite{Jacobson}. If there is no restriction on the dimension of $L$ or on the characteristic of the field, then $p$-bounded nilpotency follows from \cite{Higman}: $c(L) \leq h_p$. As already mentioned in the introduction, the explicit upper bound $c(L) \leq \frac{(p-1)^{2^{p-1}-1}-1}{p-2}$ was given in \cite{KrekninKostrikin}. And the upper bound $c(L) \leq |X|^0 + |X|^1 + \cdots + |X|^{2^{|X|}-2}$ was given in \cite{KhukhroSupport}.

\paragraph{Algebras of derivations acting without constants.} Suppose that a finite-dimensional Lie algebra $L$ of characteristic zero admits a nilpotent algebra of derivations. After extending the scalars, the weight-decomposition of $L$ is a $(\Z^m,+)$-grading of $L$. If the derivation algebra also acts without constants, then the support $X$ is contained in $\Z^m \setminus \{ 0 \}$. So we have arrived at \emph{Example \ref{ExampleTorsionFree}}. The nilpotency was proven in \cite{Jacobson}. In \cite{BergenGrzeszczuk}, an ordering on $\Z^m$ that is stronger than a good-ordering was used to obtain an implicit upper bound of the form $$ |X|^2 ( |X|! e + |X| ) .$$

\paragraph{Automorphisms with eigenvalues of infinite order.} Consider a finite-dimensional Lie algebra $L$ with an automorphism $\alpha$. After extending the scalars, we see that the eigenspace decomposition of $L$ with respect to $\alpha$ is a $(\F^\times,\cdot)$-grading of $L$. If none of the eigenvalues is a root of unity, then we have arrived at \emph{Example \ref{ExampleOrder}}. The nilpotency of such an algebra was proven in \cite{Jacobson}. 

\paragraph{The properties of Lie algebras with periodic derivations.} We offer some final remarks on how our results about arithmetically-free sets, the existence of periodic regular transformations, and the co-class conjectures of Leedham-Green and Newman are related, \cite{LGN}. \newline

In \cite{ShalevZelmanov}, Shalev and Zel'manov reduced conjecture $C$ (pro-$p$ groups of finite co-class are solvable) to proving the existence of a periodic automorphism of a Lie ring that fixes only the trivial element. Shalev later reduced conjecture $A$ (every $p$-group of co-class $r$ has a normal subgroup of class at most $2$ and $(p,r)$-bounded index) to proving the existence of a periodic derivation of a Lie ring of order $p-1$, \cite{ShalevCoclass}. \newline

We have already mentioned that if a finite-dimensional Lie algebra $L$ of characteristic zero admits a periodic derivation (or, more generally, a non-singular derivation), then $L$ is nilpotent. Indeed: Burde and the author have classified all finite-dimensional, complex Lie algebras with a periodic derivation; they turn out to be nilpotent of class at most two, \cite{BurdeMoensPeriodic}. But if the characteristic of the algebra is non-zero, then the situation is very different: there are finite-dimensional, \emph{simple} Lie algebras of characteristic $p > 0$, admitting a periodic derivation (of order $p^m-1$). (Such algebras have in fact been classified by Benkart-Kostrikin-Kuznetsov for $p > 7$, \cite{BKK}.) These examples illustrate that in characteristic $p > 0$, the derivation's eigenvalues do not always form an arithmetically-free set. In view of this observation, one must then ask the following question (which Shalev did without using the current terminology, \cite{ShalevOrder}):

\begin{openproblem}[Shalev] For which prime $p$ and natural $n$ is the set $\mathbb{X}_{n,p} := \{ x \in \overline{\mathbb{F}}_p | x^n = 1 \}$ of $n$'th roots of unity arithmetically-free in the group $(\overline{\mathbb{F}}_p,+)$? \end{openproblem}

If $n = p-1$, then \emph{Example \ref{ExampleOrder}} shows that $\mathbb{X}_{p-1,p}$ is indeed arithmetically-free in $(\overline{\mathbb{F}}_p,+)$. A more systematic study of this problem by Shalev and Mattarei appears in \cite{ShalevOrder,Mattarei,Mattarei2}.

\paragraph{Criterium.} In order to help us test whether $\mathbb{X}_{n,p}$ is arithmetically-free, we introduce some integer-valued invariants. To a given a polynomial $P$ over a field $\F$, 
and $r \in \N$, we associate the scalar $$ \Delta(P,r) := \prod_{\substack{a,b \in \overline{\F}\\P(a) = P(b) = 0}} \sum_{1 \leq t \leq r} P(a + t b) \in \overline{\F}. $$ If $\Delta(P,r) \neq 0$, then the set of roots $X := \{ a \in \overline{\F} | P(a) = 0 \}$ is an arithmetically-free subset of $(\overline{\F},+)$ and $\nu(X) \leq r$. Let us consider in particular the polynomial $P(z) := z^n - 1$. We then see that $\Delta(z^n-1,r)$ is a power of the circulant determinant
$$ \Delta_{n,r} := \det \operatorname{Circ}_{0 \leq t \leq n-1} \left( \binom{n}{t} (1^{n-t} + 2^{n-t} + \cdots + r^{n-t} ) \right) \in \Z. $$ 

\begin{example}[$\mathbb{X}_{6,p}$] Let $\omega$ be a primitive sixth root of unity. Then $\omega^2 + 1 = \omega$, so that $\mathbb{X}_{6,p}$ is not sum-free. Indeed: $\Delta_{6,1} = 0$. Since the set $\mathbb{X}_{6,2}$ is not sum-free, it is also not arithmetically-free. But by comparing the prime decompositions of the determinants $\Delta_{6,r}$, for $2 \leq r \leq 6$, we obtain: $\nu(\mathbb{X}_{6,7}) \leq 6$, $\nu(\mathbb{X}_{6,283}) \leq 4$, $\nu(\mathbb{X}_{6,113}) \leq 3$, and $\nu(\mathbb{X}_{6,p}) = 2$ for every remaining odd prime $p$. In particular: $\mathbb{X}_{6,p}$ is arithmetically-free in $(\overline{\mathbb{F}}_p,+)$, iff $p \neq 2$. \end{example}

The criterium for $r := 1$ and $P(z) := z^n-1$ yields \emph{sum-free} sets: $\nu(\mathbb{X}_{n,p}) = 1$ (cf. \emph{Example} \ref{ExampleOrder} and \ref{ExampleSumFree}). It corresponds with the theorem:

\begin{theorem}[Kostrikin-Kuznetsov] Consider a Lie algebra $L$ over a field of characteristic $p \geq 0$ with a derivation of order $n \in \N$. If $p \not | \Delta_{n,1}$, then $L$ is abelian. \end{theorem}

\section{Automorphisms of groups satisfying an identity}

The following definition is inspired by split automorphisms $f$ of finite order $n$ (already mentioned in the introduction), and by $n$-abelian groups, which have an endomorphism $f$ satisfying the identity $1_G = f(x) \cdot x^{-n},$ for all $x \in G$.

\begin{definition} Consider a group $(G,\cdot)$ and an endomorphism $f$ of $G$ satisfying the identity \begin{equation} \forall x \in G : \phantom{ooo} 1_G = \prod_{1 \leq t \leq k} f^{m_t}(x^{n_t}), \tag{$\ast$} \end{equation} where $k \in \N$, $m_1,\ldots,m_k \in \N \cup \{ 0 \}$, and $n_1,\ldots,n_k \in \Z$ are fixed. Then we associate to this identity and a field $\F$ the polynomial $$r_{\F}(z) := \sum_{1 \leq t \leq k} n_t \cdot z^{m_t} \in \F[z].$$ We say that the identity is \emph{arithmetically-free over} $\F$, iff the root set of $r_{\F}(z)$ is an arithmetically-free subset of $(\overline{\F}^\times,\cdot)$. We say that the identity is \emph{irreducible over} $\F$, iff $r_{\F}(z)$ is an irreducible polynomial over $\F$ satisfying $r_{\F}(0) \cdot r_{\F}(1) \neq 0$. The \emph{degree} of the identity is $\max \{ | m_s - m_t | | 1 \leq s, t \leq k \}$. \end{definition}

\begin{example} \label{ExampleIdentitiesAF} The split identity, $1_G = x \cdot f(x) \cdot f^2(x) \cdots f^{p-1}(x),$ with $p \in \mathbb{P}$, is arithmetically-free over every field $\F$ of characteristic unequal to $p$. In this case, $r_{\F}(z) = \Phi_p(z)$, the cyclotomic polynomial of degree $p-1$. An identity for which $r_{\Q}(z) $ divides a natural power of $ \Phi_{n_1}(z) \cdots \Phi_{n_k}(z)$, with $n_1,\ldots,n_k$ relatively prime and unequal to $1$, is arithmetically-free over $\Q$. \end{example}

\begin{remark} If $f$ is an \emph{automorphism}, then it makes sense to draw the $m_i$ from $\Z$. The corresponding $r_{\F}(z)$ is then a Laurent-polynomial over $\F$. But by applying $f$ repeatedly to the identity, we obtain a new identity $1_G = \prod_t f^{m_t + m_0} (x^{n_t})$ with corresponding $\widetilde{r}_{\F}(z) = z^{m_0} \cdot r_{\F}(z)$. In particular: if $f$ satisfies an arithmetically-free identity over $\F$ with coefficients $m_i \in \Z$, then it also satisfies an arithmetically-free identity over $\F$ of the same degree with all coefficients $\widetilde{m}_i$ in $\N \cup \{ 0 \}$. 
\end{remark}

If the root set $X$ of a polynomial over $\F$ is arithmetically-free in $(\overline{\F}^\times,\cdot)$, then clearly $0,1 \not \in X$. If the polynomial is irreducible, then the converse implication holds.

\begin{lemma} \label{LemmaIrreducible} Let $\F$ be either $ \Q $ or $ \F_p $. If $r(z) \in \F[z]$ is irreducible and $r(0) \cdot r(1) \neq 0$, then the root set $X$ of $r(z)$ is arithmetically-free in $(\overline{\F}^\times,\cdot)$. \end{lemma}

\begin{proof} 
We may assume that $r(z)$ is monic. Since $r(0) \neq 0$, we get $X \subseteq \overline{\mathbb{F}}^\times$. If $X$ contains no root of unity, then $X$ is arithmetically-free by \emph{Example}  \ref{ExampleOrder}. Else, $X$ contains some primitive $m$'th root of unity. Since $r(z)$ is irreducible, all of its roots are primitive $m$'th roots of unity. Since $r(1) \neq 0$, we have $m \neq 1$. By \emph{Example}  \ref{ExampleUnits}, we may conclude that $X$ is arithmetically-free.\end{proof}


\begin{lemma}[Linearisation] \label{LemmaLinearisation} Consider a group $(G,\cdot)$ and an automorphism $f$ of $G$ satisfying the identity $(\ast)$. Let $(N_i)_{i \in \N}$ be a characteristic $N$-series of $G$. Then the corresponding Lie ring $L$ inherits an endomorphism $\map{\overline{f}}{L}{L}$ satisfying the linearised identity \begin{equation} \forall v \in L : \phantom{ooo} 0_{L} = \sum_{1 \leq t \leq k} n_t \cdot \overline{f}^{m_t}(v) . \tag{$\dag$}\end{equation} 
If the additive group of $L$ is isomorphic to $(\Z^h,+)$ for some $h \in \N$, then $L$ naturally embeds into a $(\overline{\Q},\cdot)$-graded Lie algebra $\widetilde{L}$ of dimension $h$ with support $X$ satisfying $r_{\Q}(X) = \{ 0 \}$. Similarly, if the additive group of $L$ is isomorphic to $(\Z_p^h,+)$ for some $h \in \N$, then $L$ naturally embeds into a $(\overline{\F}_p,\cdot)$-graded Lie algebra $\widetilde{L}$ of dimension $h$ with support $X$ satisfying $r_{\F_p}(X) = \{ 0 \}$. 
\end{lemma}

\begin{proof} We recall that additive group of the Lie ring $L$ is defined by $L := \bigoplus_{i \in \N} N_i / N_{i+1}.$ Since the series is characteristic, each abelian section $N_i / N_{i+1}$ inherits a group endomorphism $\map{\overline{f}_i}{N_i/N_{i+1}}{N_i/N_{i+1}}$ that satisfies the same identity. The $\overline{f}_i$ then extend linearly to a Lie-endomorphism $\map{\overline{f}}{L}{L}$ of $L$. Since $(\dag)$ holds on all homogeneous elements of $L$ (by definition), it also holds on all of $L$ (by linearity). \newline

Now suppose that the additive group is isomorphic to $(\Z^h,+)$. Then we obtain $\widetilde{L}$ from $L$ by a simple extension of the scalars (from $\Z$ to $\overline{\Q}$), and we obtain a Lie-endomorphism $\map{\widetilde{f}}{\widetilde{L}}{\widetilde{L}}$ satisfying the same linearised identity $(\dag)$ as does $\overline{f}$. The eigenspace decomposition of $\widetilde{L}$ w.r.t. the Lie-endomorphism $\widetilde{f}$ is then a $(\overline{\Q},\cdot)$-grading, and the support is contained in the root set of $r_{\Q}(z)$. The other case is completely analogous. \end{proof}

\begin{theorem} Consider a finitely-generated, torsion-free, nilpotent group $(G,\cdot)$. If one of its automorphisms satisfies an arithmetically-free identity of degree $m$ over $\Q$, then $G$ is nilpotent of class at most $m^{2^m}.$ \label{TheoremTorsionFree} \end{theorem}

\begin{proof} We let $N_i$ be the isolator of the $i$'th term in the lower central series of $G$: $I(\gamma_i(G)) := \{ x \in G | \exists b \in \N : x^b \in \gamma_{i}(G) \}$. The Lie ring $L$ associated to this series is known to be nilpotent with $c(G) = c(L)$. Moreover, the additive group of $L$ is $(\Z^h,+)$, where $h$ is the Hirsch-length of $G$. We may then combine lemma \ref{LemmaLinearisation} and corollary \ref{CorollaryGrading}.
\end{proof}

If the group has torsion, then the situation is more delicate. Nonetheless, we obtain:

\begin{theorem} Consider a finitely-generated group $(G,\cdot)$ that is residually-(finite $p$). If one of its automorphisms satisfies an identity that is arithmetically-free over $\F_p$, then $G$ is linear. If $G$ is also a torsion group, then $G$ is finite. \label{TheoremAutomorphismP} \end{theorem}

Examples of Golod show that the condition on the identity cannot be dropped. The following argument is close to Zel'manov's solution for the restricted Burnside problem, \cite[Theorem $1.2$]{Zelmanov}:

\begin{proof} We consider the $p$-dimension series $(N_i)_{i \in \N}$ of $G$ and the associated Lie ring $L$ (which is in fact a restricted Lie algebra over $\F_p$). Let $m$ be the degree of the identity. 
Note that the finite quotients $ L / \bigoplus_{s \geq t} N_s / N_{s+1}$ are of class $\leq m^{2^m}$, by Lemma \ref{LemmaLinearisation} and corollary \ref{CorollaryGrading}. Then also $$\gamma_{m^{2^m} + 1}(L) \subseteq \bigcap_{t \in \N} \bigoplus_{s \geq t} N_s / N_{s+1} = \{ 0 \}.$$ 


An argument of Lazard now implies that the pro-$p$ completion $\widehat{G}$ of $G$ is $p$-adic analytic, and therefore linear. Since $G$ is residually-(finite $p$), it embeds into $\widehat{G}$. So $G$ is finitely-generated and linear. If $G$ is also periodic, then the theorem of Burnside-Jordan-Schur implies that $G$ is finite. 
\end{proof}

One would hope to also obtain, as in the positive solution of the restricted Burnside problem, an upper bound for the cardinality of $G$ (depending only on $p$, the number of generators, and the degree of the identity). If $G$ is $d$-generated and has prime exponent, then $c(G) \leq m^{2^m}$ (simply use the lower central series, as in theorem \ref{TheoremTorsionFree}), so that $|G|$ is bounded by a function of $d,p$ and $m$. But the following rather trivial example shows that such an upper bound for $|G|$ cannot exist with arbitrary $\operatorname{exp}(G)$.

\begin{example} \label{ExampleExponent} For each odd prime $p$ and natural $l$ we consider the cyclic group $(G_{p,l},\cdot)$ of order and exponent $p^l$. Since the group is abelian, the map $\mapl{f}{G_{p,l}}{G_{p,l}}{x}{x^{-1}}$ is an automorphism satisfying the identity $$1_{G_{p,l}} = x \cdot f(x).$$ The corresponding polynomial is $r_{\F_p}(z) = z + 1$ with root set $X = \{ -1 \}$. The latter is arithmetically-free, since $p$ is odd. So we have $1$-generated groups of arbitrarily large cardinality, satisfying an arithmetically-free identity of degree $1$. \end{example}

The assumptions of the restricted Burnside problem are sufficient to give a nice upper bound:

\begin{corollary} \label{CorollaryBoundP} Consider a residually-finite $p$-group $(G,\cdot)$ of exponent $l$, generated by $d$ elements. If one of its automorphisms satisfies an arithmetically-free identity of degree $m$ over $\F_p$, then $G$ is finite, and $$|G| \leq l^{ d^{m^{2^m}+1} }.$$
\end{corollary}

\begin{proof} We assume the notation of Theorem \ref{TheoremAutomorphismP}, which allows us to assume that $G$ is finite. (Alternatively, we may use Zel'manov's results for groups that are infinitesimally-PI.) Then $|G| = |L| = p^{\dim_{\F_p}(L)}.$ Let $a_1,\ldots,a_d$ be the generating set of $G$. Then $L$ is generated \emph{as a restricted algebra}, by the cosets $a_1 \cdot N_2,\ldots, a_d \cdot N_2$. Let $L^\ast$ be the \emph{ordinary} subalgebra of $L$ that is generated by these cosets. Since $c(L^\ast) \leq c(L) \leq m^{2^m} $, we get $\dim_{\F_p}(L^\ast) \leq d^{m^{2^m}+1}$. An argument of Bahturin now implies $\dim_{\F_p}(L) \leq \log_p(l) \cdot \dim_{\F_p}(L^\ast)$, \cite[Prop. $2$, p. $17$]{Bahturin}, \cite[Lem. $5.3$]{ShalevIdentities}. This finishes the proof.
\end{proof}




\begin{openproblem} If $G$ is a finite $p$-group satisfying an arithmetically-free identity of degree $m$ over $\F_p$, then $c(G) \leq m^{2^m}$. \end{openproblem}


Not every identity of automorphisms lends itself to this approach.

\begin{remark}(``Classical identities that are \emph{not} arithmetically-free.'') $(1)$. Let $y$ be an $m$-\emph{Engel} element of a group $(G,\cdot)$, and let $\mapl{f}{G}{G}{x}{y^{-1} \cdot x \cdot y}$ be the conjugation by $y$. Define the sequence of maps $\map{(w_n)_{n \in \N}}{G}{G}$ by $w_1(x) := [x,y] := x^{-1} \cdot f(x)$ and $w_{n+1}(x) := w_1(w_n(x))$. Then $f$ satisfies the identity $1_G = w_m(x),$ and the corresponding polynomial is $r_{\F}(z) = \pm (z-1)^m$, so that $X(r_{\F}(z)) = \{ 1 \}$. In particular: this identity is never arithmetically-free over any field. 
$(2)$. The identity $1_G = x^{p^l}$ is not arithmetically-free over $\F_p$ since $r_{\F_p}(z) := p^l$ is the zero element of $\F_p[z]$. \end{remark}

\paragraph{Thanks.} The author would like to thank his host, Efim Zel'manov, and Lance Small for their hospitality during the Erwin Schr\"odinger Research Programme (\emph{Representations and gradings of solvable algebras}: $2013-2015$, $J3371-N25$) at the University of California, San Diego. He would also like to express his gratitude to the Erwin Schr\"odinger International Institute for Mathematical Physics where preliminary work for the research was done (\emph{Lie algebras: deformations and representations}). Finally, he thanks the Geometric and Analytic Group Theory-group of the University of Vienna.

\end{document}